\documentclass[11pt]{amsart}
\usepackage{amsmath,amssymb,amsthm,mathtools}
\usepackage[margin=25mm]{geometry}
\usepackage{enumitem}
\usepackage{xcolor}
\usepackage[hypertexnames=false,colorlinks=true,linkcolor=blue!50!black,citecolor=blue!50!black,urlcolor=blue!50!black]{hyperref}

\theoremstyle{plain}
\newtheorem{mainthm}{Theorem}

\newtheorem{thm}{Theorem}[section]
\newtheorem{prop}[thm]{Proposition}
\newtheorem{lem}[thm]{Lemma}
\newtheorem{cor}[thm]{Corollary}
\theoremstyle{definition}
\newtheorem{dfn}[thm]{Definition}
\newtheorem{ex}[thm]{Example}

\theoremstyle{remark}
\newtheorem{rem}[thm]{Remark}
\newtheorem{nota}[thm]{Notation}

\numberwithin{equation}{section}

\newcommand{\RR}{\mathbb{R}}
\newcommand{\XX}{\mathfrak{X}}
\newcommand{\Xmap}{\mathcal{X}}
\newcommand{\Elast}{\mathsf{Elast}}
\newcommand{\Elastc}{\mathsf{Elast}_{\mathrm{curve}}}
\newcommand{\Rpt}{\mathbf{R}}      
\newcommand{\Rfour}{\mathcal{R}}   
\newcommand{\Qop}{\mathcal{Q}}     
\DeclareMathOperator{\id}{id}
\DeclareMathOperator{\Path}{Path}
\DeclareMathOperator{\ad}{ad}
\DeclareMathOperator{\sect}{sec}
\DeclareMathOperator{\Ric}{Ric}
\DeclareMathOperator{\tr}{tr}
\DeclareMathOperator{\supp}{supp}

\title[Curvature and second variation in elastic diffeology]{Curvature, second variation and Jacobi fields\\ in elastic diffeology}
\author[Y. Shiobara]{Yusuke Shiobara}
\address{%
  Department of Science and Technology,
  Graduate School of Medicine, Science and Technology,
  Shinshu University,
  Matsumoto, Nagano 390-8621, Japan}
\email{24hs603d@shinshu-u.ac.jp}
\subjclass[2020]{Primary 58A40; Secondary 18F40, 53B20, 58B20, 58E10}
\keywords{diffeology, tangent category, curvature, second variation, Jacobi field, index form}

\begin{document}

\begin{abstract}
We develop curvature theory on Blohmann's elastic diffeological spaces. We introduce curvature-tensorial connections, whose curvature factors uniquely through $T_3M$, and show that this property holds on finite-dimensional manifolds and is preserved on mapping spaces. Using a plotwise commutator identity and the elasticity axioms, we establish the classical curvature symmetries for curvature-tensorial Levi-Civita connections. We define sectional curvature and a curvature form on the algebraic exterior square, and demonstrate an obstruction to the usual Ricci trace construction: on $C^\infty(S^1,S^2)$, the relevant curvature endomorphism extends boundedly to the $L^2$ completion but is not trace class. We express the second variation of the energy by the index form and characterize the null space of the index form by Jacobi fields. Given a geodesic flow, we construct Jacobi fields with prescribed initial data along its flow lines; they are unique when the real line is a curve object for the space. Finally, non-positive sectional curvature excludes conjugate points, while a negative index direction obstructs length minimization; the latter result does not require the curve-object property and applies to mapping spaces such as free loop spaces of complete Riemannian manifolds.
\end{abstract}

\maketitle
\tableofcontents

\section{Introduction}\label{sec:intro}

Diffeology was introduced by Souriau~\cite{Sou}, and its systematic development is due to Iglesias-Zemmour~\cite{PIZ13}. It generalizes smooth manifolds: the category of diffeological spaces contains the manifolds as a full subcategory and is complete, cocomplete and cartesian closed, so that singular quotients and many infinite-dimensional spaces are treated on the same footing.
Differential forms, de~Rham theory and homotopy theory have been developed extensively in this setting \cite{PIZ13,Kihara}, but Riemannian geometry is less developed, partly because the various notions of tangent space on diffeological spaces do not agree in general~\cite{Taho}.
Riemannian metrics were introduced by Iglesias-Zemmour~\cite{PIZ25} without reference to tangent spaces and interpreted through Blohmann's tangent functor in~\cite{KSS}; Goldammer and Welker~\cite{GW} studied metrics and Levi-Civita connections on Vincent's tangent spaces~\cite{Vin}.
For connections, geodesics and variational theory it is convenient to work in a tangent category in the sense of Rosick\'y~\cite{Ros} and Cockett--Cruttwell~\cite{CC14}, where connections and their curvature are available~\cite{CC17,LW}.
Blohmann's tangent functor $T$, the left Kan extension of the tangent functor of Euclidean domains, makes the category $\Elast$ of elastic diffeological spaces such a tangent category~\cite{Bloh23}.

In \cite{Shio} we used this structure to develop Riemannian geometry on elastic spaces. We introduced connections following \cite{CC17,LW}, Riemannian metrics in the sense of \cite{KSS}, Levi-Civita connections, geodesics and the geodesic flow, and proved a variational characterization of geodesics: if the geodesic flow fixes the zero section, the critical paths of the energy functional are exactly the paths satisfying the geodesic equation on $[0,1]$ \cite[Theorem~A]{Shio}. The curvature morphism of \cite{CC17} was recalled in \cite[\S5]{Shio}, but its interaction with a Riemannian metric was not treated there. The present paper takes up this question and develops the second-order variational theory.

The classical proofs of the results below are local: they use bump functions and local frames, charts, the exponential map, and the local existence and uniqueness theory of ordinary differential equations. None of these is available on a general elastic space. Smooth functions may be scarce, a tangent vector need not be the value of a global vector field, there are no charts, and since $\Elast$ contains all convenient vector spaces, $\RR$ is not a curve object in $\Elast$ \cite[\S6.1]{Shio}. The aim of this paper is to show that the second-order theory nevertheless goes through, and to determine which structure each step actually requires. Our main contributions are the following.
\begin{enumerate}[label=\textup{(\roman*)}]
\item A pointwise curvature tensor and the four classical curvature symmetries at every point, obtained without localization (Theorem~A below).
\item The second variation formula, proved with plots of the path space in place of charts or an exponential map; its key input is a commutation formula for covariant derivatives along plots, which has no bracket term (Theorem~B).
\item Jacobi fields, realized as solutions of the linearized geodesic equation on $T^2M$. Existence requires only a geodesic flow, while uniqueness requires $\RR$ to be a curve object for $T^2M$; this explains why the definition of $\Elastc$ in \cite{Shio} imposes the curve-object condition on every iterated tangent bundle (Theorem~C).
\item Two phenomena without finite-dimensional counterpart: on $C^\infty(S^1,S^2)$ the Ricci curvature cannot be defined as a trace, in line with Freed's observations for loop groups \cite{Fre88} (Example~\ref{ex:noric}), and the non-minimality theorem applies to free loop spaces of complete Riemannian manifolds, although these spaces are not known to belong to $\Elastc$ (Theorem~D(3)).
\end{enumerate}

For~(i), the curvature of a vertical connection $K$ is a morphism $C_K\colon T^3M\to TM$, and the formula $R^K(X,Y,Z)=\nabla_X\nabla_YZ-\nabla_Y\nabla_XZ-\nabla_{[X,Y]}Z$ of \cite[Proposition~5.4]{Shio} evaluates it on global vector fields only. Since localization is not available, identities between global vector fields need not say anything about individual tangent vectors. We deal with this in three steps. First, we isolate the \emph{curvature-tensorial} connections, for which $C_K$ factors through a trilinear morphism $\Rpt^K\colon T_3M\to TM$ (Section~\ref{sec:tensorial}). The projection $\Pi\colon T^3M\to T_3M$ through which $C_K$ factors is a subduction for every elastic space (Proposition~\ref{prop:Pi-subd}), so $\Rpt^K$ is unique and automatically antisymmetric in its first two arguments. Second, we replace global vector fields by plots: the commutator of two covariant derivatives along a plot is given by the curvature morphism (Theorem~\ref{thm:commute}), with no bracket term. Third, Blohmann's elasticity axiom (E1) implies that any finite family of tangent vectors at a point consists of coordinate vectors of a single plot, together with fields along that plot taking any finitely many prescribed values at the base point (Lemma~\ref{lem:joint}). Every identity that can be tested on coordinate fields of plots and on fields along plots therefore holds at every point and for all tangent vectors.

\begin{mainthm}[Theorem~\ref{thm:S}]
Let $(M,g)$ be an elastic Riemannian diffeological space and $K$ a curvature-tensorial Levi-Civita connection for $g$. Then for every $x\in M$ the quadrilinear form $\Rfour_x(v,w,z,y)=g_x(\Rpt_x(v,w,z),y)$ satisfies the four classical symmetries of the Riemann tensor on $T_xM$.
\end{mainthm}

Every vertical connection on a finite-dimensional manifold is curvature-tensorial, and curvature-tensoriality passes to mapping spaces $C^\infty(\Sigma,N)$ (Proposition~\ref{prop:tensorial}); in particular Theorem~A applies to the mapping spaces with the $L^2$ metric of \cite[\S\S7.3--7.4]{Shio}.

Consequently the sectional curvature and the curvature form on $\Lambda^2T_xM$ are defined (Section~\ref{sec:sec}). A Ricci tensor cannot in general be obtained by tracing the curvature: already on $C^\infty(S^1,S^2)$ the relevant operator on the Hilbert completion is not of trace class (Example~\ref{ex:noric}).

In Section~\ref{sec:var} we fix a curvature-tensorial Levi-Civita connection $K$. Let $\XX_0(\gamma)$ be the space of fields along a path $\gamma$ vanishing at $t=0,1$, and let
\[
I(V,W)=\int_0^1\Bigl[g(\nabla_tV,\nabla_tW)-\Rfour(V,\dot\gamma,\dot\gamma,W)\Bigr]dt
\]
be the index form of a path $\gamma\in\Path(M;x,y)$ satisfying the geodesic equation on $[0,1]$.

\begin{mainthm}[Theorem~\ref{thm:hessian}, Corollary~\ref{cor:2ndvar}]
The energy functional is smooth on $\Path(M;x,y)$. Let $\gamma\in\Path(M;x,y)$ satisfy the geodesic equation on $[0,1]$. For every $2$-parameter variation plot $P$ of $\gamma$ with variation fields $V_1,V_2$,
\[
\partial_1\partial_2(E\circ P)(0)=I\bigl(V_1(0,\cdot),V_2(0,\cdot)\bigr).
\]
In particular $\frac{d^2}{ds^2}\big|_0E(P(s))=I(W,W)$ for every variation plot $P$ of $\gamma$ with variation field $W$, and $I$ is symmetric.
\end{mainthm}

For Jacobi fields and for minimality we use the hypotheses (SH1)--(SH3) of Definition~\ref{dfn:SH}: a full connection whose vertical part is a curvature-tensorial Levi-Civita connection, a geodesic flow fixing the zero section, and $M\in\Elastc$. Apart from curvature-tensoriality, (SH1) and (SH2) are the hypotheses of \cite[Theorem~A]{Shio}.

\begin{mainthm}[Theorem~\ref{thm:jacobi}]
Assume \textup{(SH1)--(SH3)}. Along every geodesic $\gamma$ with $\gamma(0)=x$ and for every $a,b\in T_xM$ there is a unique Jacobi field $J$ with $J(0)=a$ and $\nabla_tJ(0)=b$.
\end{mainthm}

Existence uses the effectiveness of the connection and a geodesic flow; uniqueness uses the curve-object condition for $T^2M$ (Remark~\ref{rem:why-all-n}).

\begin{mainthm}[Corollary~\ref{cor:conj}, Theorems~\ref{thm:nonpos} and \ref{thm:nonmin}]
Let $\gamma\in\Path(M;x,y)$ satisfy the geodesic equation on $[0,1]$.
\begin{enumerate}[label=\textup{(\arabic*)}]
\item The null space of $I$ consists of the fields in $\XX_0(\gamma)$ satisfying the Jacobi equation on $[0,1]$; hence $I$ is weakly non-degenerate \textup{(}modulo fields vanishing on $[0,1]$\textup{)} if and only if $\gamma(1)$ is not conjugate to $\gamma(0)$.
\item If $\sect\le0$, then $I(W,W)\ge\int_0^1\|\nabla_tW\|^2dt$, and no point $\gamma(t_0)$ with $0<t_0\le1$ is conjugate to $\gamma(0)$.
\item Under \textup{(SH1)} and \textup{(SH2)}, if $I(W,W)<0$ for some $W\in\XX_0(\gamma)$, then $\gamma$ does not minimize length among paths from $x$ to $y$.
\end{enumerate}
\end{mainthm}

Part~(3) does not require (SH3); its application to mapping spaces is Remark~\ref{rem:mapping-var}.

Thus several local ingredients familiar from Morse theory---smoothness of the energy, its first and second variation, and the Jacobi-field description of the null space---are available in elastic diffeology. We do not prove a Morse lemma, an index theorem, or a general local-minimality theorem.

\subsection*{Organization}
Section~\ref{sec:prelim} collects, with precise references, the definitions and results of \cite{Shio} used in this paper. Section~\ref{sec:tensorial} introduces curvature-tensorial connections. Section~\ref{sec:commute} proves the commutation formula along plots and its consequences. Section~\ref{sec:sym} establishes the symmetries of the curvature, globally and pointwise. Section~\ref{sec:sec} treats sectional curvature, the curvature form and the obstruction to defining Ricci curvature by a trace. Section~\ref{sec:var} contains the second variation formula, Jacobi fields and the index form.

\subsection*{Conventions}
We follow the conventions of \cite{Shio}; in particular composition is written in applicative order, every diffeological space is elastic unless stated otherwise, and Riemannian metrics are definite.

\subsection*{Acknowledgements}
I would like to express my sincere gratitude to my supervisor, Keiichi Sakai, for his valuable guidance and continuous support since my undergraduate years.
I am also grateful to my sub-advisor, Katsuhiko Kuribayashi, for many insightful discussions and much helpful advice.
This work was supported by JST SPRING, Grant Number JPMJSP2144, at Shinshu University.

\subsection*{Declaration of generative AI in the writing process}
During the preparation of this manuscript, the author used ChatGPT (OpenAI) and Claude (Anthropic) for language editing, for checking proofs, and for discussing the organization of the material.
All mathematical results, proofs, and scientific conclusions were verified by the author, who takes full responsibility for the content of this publication.

\section{Preliminaries}\label{sec:prelim}

\subsection{Tangent structure}
By \cite[Theorem~2.17]{Shio}, Blohmann's tangent functor $T$ makes $\Elast$ a Cartesian tangent category with scalar $\RR$-multiplication, with structure maps $\pi,0,+,\kappa,\tau,\lambda$. Each fibre of $\pi_M\colon TM\to M$ is a real vector space \cite[Remark~2.18]{Shio}, and likewise for the two bundle structures $\pi_{TM},T(\pi_M)\colon T^2M\to TM$; hence the difference of two morphisms into $TM$ lying over the same morphism into $M$ is defined. Besides naturality of $\pi$ and $\tau$ we use the identities
\begin{equation}\label{eq:pi-tau}
\pi_{TM}\circ\tau_M=T(\pi_M),\qquad T(\pi_M)\circ\tau_M=\pi_{TM},
\end{equation}
the second of which follows from the first and $\tau_M^2=\id$. We shall mostly apply them with $M$ replaced by $TM$.

A \emph{vector field along} a smooth map $f\colon N\to M$ is a smooth map $V\colon N\to TM$ with $\pi_M\circ V=f$ \cite[Definition~3.21]{Shio}; we usually say simply \emph{field along} $f$. For a vector field $A$ on $N$, $T(f)\circ A$ is a field along $f$. For a domain $U\subseteq\RR^n$ we write $\partial_1,\dots,\partial_n$ for the coordinate vector fields; they have constant components, so that
\begin{equation}\label{eq:flip-coord}
T(\partial_j)\circ\partial_i=\tau_U\circ T(\partial_i)\circ\partial_j
\end{equation}
(see the proof of \cite[Lemma~6.27]{Shio}). For a plot $P\colon U\to M$ we write $\partial_iP:=T(P)\circ\partial_i$ for the coordinate fields along $P$; in the notation of \cite[Remark~2.11]{Shio}, $\partial_iP(u)=[P,u,e_i]$.

For a smooth map $F\colon X\times Y\to Z$ the partial tangent $T_{(2)}F\colon X\times TY\to TZ$ is defined in \cite[\S2.2.3]{Shio}. Since $T$ of a constant map is the zero section at its value,
\begin{equation}\label{eq:partial-tangent}
T_{(2)}F(x,\xi)=T\bigl(F(x,-)\bigr)(\xi),\qquad x\in X,\ \xi\in TY .
\end{equation}

\subsection{Connections and covariant derivatives}
Let $K\colon T^2M\to TM$ be a vertical connection \cite[Definition~3.1]{Shio}. By \cite[Remark~3.2]{Shio}, $K$ is a linear bundle morphism over $\pi_M$ both for $\pi_{TM}$ and for $T(\pi_M)$; in particular it maps the zero sections of these two bundles to zero:
\begin{equation}\label{eq:K-zero}
K\circ0_{TM}=0_M\circ\pi_M,\qquad K\circ T(0_M)=0_M\circ\pi_M .
\end{equation}
For a field $V$ along $f\colon N\to M$ and $A\in\XX(N)$, the covariant derivative along $f$ is $\nabla^K_AV:=K\circ T(V)\circ A$ \cite[Definition~3.22]{Shio}. Since the connection will always be fixed, we write $\nabla_A$ for $\nabla^K_A$, $\nabla_i$ for $\nabla_{\partial_i}$, and $\nabla_t$ for $\nabla_{\partial_t}$.

\begin{lem}\label{lem:along}
Let $f\colon N\to M$ be smooth, $A\in\XX(N)$ and $V,V'$ fields along $f$.
\begin{enumerate}[label=\textup{(\arabic*)}]
\item $\nabla_A(V+V')=\nabla_AV+\nabla_AV'$ and $\nabla_A(0_M\circ f)=0_M\circ f$.
\item If $j\colon N'\to N$ is smooth and $A'\in\XX(N')$ satisfies $T(j)\circ A'=A\circ j$, then $(\nabla_AV)\circ j=\nabla_{A'}(V\circ j)$.
\item For $\varphi\in C^\infty(N,\RR)$ and $\varphi V:=\kappa_M\circ\langle\varphi,V\rangle$, one has $\nabla_A(\varphi V)=(A\varphi)V+\varphi\nabla_AV$. In particular $\nabla_A(rV)=r\nabla_AV$ for $r\in\RR$, so $\nabla_A$ is $\RR$-linear on fields along $f$.
\end{enumerate}
\end{lem}

\begin{proof}
(1) As in \cite[Proposition~3.10(3)]{Shio}: $T(+_M)$ is the addition of the bundle $T(\pi_M)$, for which $K$ is additive. Moreover $T(0_M\circ f)\circ A=T(0_M)\circ T(f)\circ A$, and \eqref{eq:K-zero} applies.
(2) $K\circ T(V)\circ A\circ j=K\circ T(V)\circ T(j)\circ A'=K\circ T(V\circ j)\circ A'$.
(3) The computation in the proof of \cite[Proposition~3.10(4)]{Shio} applies verbatim with $(X,Y,f)$ replaced by $(A,V,\varphi)$: it only uses the decomposition of $T(\kappa_M)$ into partial tangents, axioms (a) and (f) of a vertical connection, and the identities $\pi_{TM}\circ T(V)\circ A=V\circ\pi_N\circ A=V$, $\pi_\RR\circ T(\varphi)\circ A=\varphi$ and $\mathrm{pr}_2\circ T(\varphi)\circ A=A\varphi$. For constant $\varphi=r$ we have $A\varphi=0$.
\end{proof}

Lemma~\ref{lem:along}(2) will be applied to slices $j(t)=(s_0,t)$ and to restrictions to open subsets; we refer to it as the \emph{restriction principle}.

\begin{lem}[{\cite[Lemma~6.27]{Shio}}]\label{lem:symmetry}
If $K$ is torsion-free and $P\colon U\subseteq\RR^n\to M$ is a plot, then $\nabla_i\partial_jP=\nabla_j\partial_iP$ for all $i,j$.
\end{lem}

The proof in \cite{Shio} is written for $n=2$ and applies verbatim to any two coordinate directions, by \eqref{eq:flip-coord}.

A \emph{full connection} $(K,H)$ \cite[Definition~3.15]{Shio} consists of $K$ and a horizontal connection $H\colon TM\times_MTM\to T^2M$ with
\begin{equation}\label{eq:H}
T(\pi_M)\circ H=\mathrm{pr}_1,\qquad \pi_{TM}\circ H=\mathrm{pr}_2,\qquad K\circ H=0_M\circ\pi_M\circ\mathrm{pr}_1 .
\end{equation}
Then $K$ is effective, i.e.\ $\Theta:=\langle\pi_{TM},T(\pi_M),K\rangle\colon T^2M\to T_3M:=TM\times_MTM\times_MTM$ is a diffeomorphism \cite[Proposition~3.16]{Shio}.

\subsection{Metrics}
Let $g$ be a Riemannian metric on $M$ \cite[Definition~4.1]{Shio} and $K$ a Levi-Civita connection for $g$ \cite[Definition~4.11]{Shio}. For fields $V,W$ along $f\colon N\to M$ and $A\in\XX(N)$ we have metric compatibility along maps \cite[Proposition~6.26]{Shio}:
\begin{equation}\label{eq:metric-along}
A\,g(V,W)=g(\nabla_AV,W)+g(V,\nabla_AW).
\end{equation}

\subsection{Curvature}
The curvature morphism of $K$ and its curvature tensor are \cite[Definitions~5.1 and 5.3]{Shio}
\[
C_K:=K\circ T(K)\circ\tau_{TM}-K\circ T(K)\colon T^3M\to TM,\qquad
R^K(X,Y,Z):=C_K\circ T^2(Z)\circ T(X)\circ Y,
\]
for $X,Y,Z\in\XX(M)$, and by \cite[Proposition~5.4]{Shio}
\begin{equation}\label{eq:R-nabla}
R^K(X,Y,Z)=\nabla_X\nabla_YZ-\nabla_Y\nabla_XZ-\nabla_{[X,Y]}Z .
\end{equation}
Following \cite[Example~5.2]{Shio}, we write points of $T^3\RR^n$ as $(v_S)_{S\subseteq\{0,1,2\}}$ with $v_\emptyset=u$: the direction $0$ is the fibre direction of $\pi_M$, the direction $1$ is the outer direction of $T^2M=T(TM)$, and the direction $2$ is the outer direction of $T^3M=T(T^2M)$. Then $\tau_{TM}$ exchanges the directions $1$ and $2$.

\subsection{Geodesics and variations}
We use the notation of \cite[\S6]{Shio}: $\Path(M;x,y)$ is the space of smooth paths $\gamma\colon\RR\to M$ with $\gamma(0)=x$ and $\gamma(1)=y$, and $\dot\gamma=T(\gamma)\circ\partial$. A path $\gamma$ is a \emph{geodesic} if $\nabla_t\dot\gamma=0_M\circ\gamma$, and it \emph{satisfies the geodesic equation on $[0,1]$} if $\nabla_t\dot\gamma(t)=0_M(\gamma(t))$ for all $t\in[0,1]$ \cite[Definition~6.9]{Shio}. For a full connection $(K,H)$ the geodesic spray is $S=H\circ\langle\id,\id\rangle$, so that $\pi_{TM}\circ S=T(\pi_M)\circ S=\id_{TM}$ and $K\circ S=0_M\circ\pi_M$ by \eqref{eq:H}.

A \emph{geodesic flow} of $(K,H)$ is a solution $\Phi\colon\RR\times TM\to M$ of the second-order dynamical system $(M,S,\id_{TM})$ \cite[Definitions~6.13 and 6.16]{Shio}; that is, with
\[
\Phi^{(1)}:=T(\Phi)\circ(\partial\times0_{TM}),\qquad
\Phi^{(2)}:=T(\Phi^{(1)})\circ(\partial\times0_{TM}),
\]
it satisfies the initial and differential conditions
\[
\Phi^{(1)}(0,v)=v,\qquad \Phi^{(2)}=S\circ\Phi^{(1)}.
\]
No uniqueness of $\Phi$ is assumed. The proof of \cite[Lemma~6.31]{Shio} gives the first two identities below, and the existence part of the proof of \cite[Theorem~6.22]{Shio}, which rests on \cite[Lemma~6.19]{Shio}, gives the third (cf.\ the proof of \cite[Proposition~6.23]{Shio}); these arguments use only that $\Phi$ is a geodesic flow:
\begin{equation}\label{eq:flow}
\Phi(0,v)=\pi_M(v),\qquad \Phi^{(1)}(0,v)=v,\qquad t\mapsto\Phi(t,v)\ \text{is a geodesic}.
\end{equation}
The flow $\Phi$ \emph{fixes the zero section} if $\Phi(s,0_M(x))=x$ for all $s\in\RR$ and $x\in M$ \cite[Definition~6.17]{Shio}. If a geodesic flow exists, then $S$ is complete, because $\Phi^{(1)}\colon\RR\times TM\to TM$ is then a solution of the parametrized dynamical system $(TM,S,\id_{TM})$. Conversely, if $M\in\Elastc$ and $S$ is complete, a geodesic flow exists, every geodesic flow fixes the zero section, and every geodesic $\gamma$ satisfies $\gamma(t)=\Phi(t,\dot\gamma(0))$ \cite[Theorem~6.22, Proposition~6.23]{Shio}.

The energy is $E(\gamma)=\frac12\int_0^1g(\dot\gamma,\dot\gamma)\,dt$, and the length of a path is
\[
L(\gamma):=\int_0^1g(\dot\gamma,\dot\gamma)^{1/2}dt
\]
\cite[Definition~6.34]{Shio}; the integrand is continuous and non-negative. By \cite[Lemma~6.35]{Shio}, for every path,
\begin{equation}\label{eq:LE}
L(\gamma)^2\le2E(\gamma),\quad\text{with equality iff }g(\dot\gamma,\dot\gamma)\text{ is constant on }[0,1].
\end{equation}

\begin{dfn}\label{dfn:variation}
Let $\gamma\in\Path(M;x,y)$ and $k\ge1$. A \emph{$k$-parameter variation plot} of $\gamma$ is a plot $P\colon U\to\Path(M;x,y)$, where $U\subseteq\RR^k$ is an open neighbourhood of $0$, with $P(0)=\gamma$. Its adjoint $\ad(P)\colon U\times\RR\to M$ is a plot of $M$, and the \emph{variation fields} of $P$ are the fields
\[
V_i:=\partial_i\ad(P),\qquad V_t:=\partial_t\ad(P),\qquad\text{i.e.}\quad V_i(u,t)=[\ad(P),(u,t),e_i],
\]
along $\ad(P)$, for $i=1,\dots,k$.
\end{dfn}
For $k=1$ and $U=(-\varepsilon,\varepsilon)$ this is a variation plot of $\gamma$ in the sense of \cite[Definition~6.24]{Shio}, since the endpoint conditions there hold automatically for a plot of $\Path(M;x,y)$, and $V_1$ is the field $V_s$ of \cite[Lemma~6.28]{Shio}. We allow arbitrary open neighbourhoods $U$ of $0$ because reparametrizations such as $(s_1,s_2)\mapsto P(s_1+s_2)$ in Corollary~\ref{cor:2ndvar} are not defined on boxes. By the restriction principle, $V_t(0,\cdot)=\dot\gamma$.

\begin{rem}\label{rem:01}
The energy only depends on $\gamma|_{[0,1]}$. Accordingly, a path satisfying the geodesic equation on $[0,1]$ is critical for $E$ \cite[Proposition~6.30]{Shio}, and conversely, under the hypotheses of \cite[Theorem~A]{Shio}, every critical path satisfies the geodesic equation on $[0,1]$ \cite[Proposition~6.33]{Shio}. No conclusion about $\gamma$ outside $[0,1]$ follows from criticality of $E$; we therefore state all variational results for paths satisfying the geodesic equation on $[0,1]$.
\end{rem}

\begin{dfn}[hypotheses]\label{dfn:SH}
Let $(M,g)$ be an elastic Riemannian diffeological space.
\begin{enumerate}[label=(SH\arabic*)]
\item $(K,H)$ is a full connection on $M$ whose vertical part $K$ is a Levi-Civita connection for $g$ and is curvature-tensorial (Definition~\ref{dfn:tensorial}).
\item $\Phi\colon\RR\times TM\to M$ is a geodesic flow of $(K,H)$ that fixes the zero section.
\item $M\in\Elastc$.
\end{enumerate}
\end{dfn}
Hypotheses (SH1) and (SH2) are the standing assumptions of \cite[\S6.6]{Shio}, that is, the hypotheses of \cite[Theorem~A]{Shio}, together with curvature-tensoriality. Under (SH3), every geodesic flow of $(K,H)$ fixes the zero section by the discussion above; so (SH2) holds as soon as $(K,H)$ has a geodesic flow, equivalently as soon as $S$ is complete. These hypotheses are not used in Sections~\ref{sec:tensorial}--\ref{sec:sec}; in Section~\ref{sec:var} they are imposed only where stated (see the beginning of that section).

\section{Curvature-tensorial connections}\label{sec:tensorial}

\subsection{Trilinearity over functions}
The following holds for every vertical connection.

\begin{prop}\label{prop:trilinear}
Let $K$ be a vertical connection on $M$. Then $R^K$ is additive in each variable, and for $X,Y,Z\in\XX(M)$ and $f\in C^\infty(M,\RR)$,
\[
R^K(fX,Y,Z)=R^K(X,fY,Z)=R^K(X,Y,fZ)=f\,R^K(X,Y,Z).
\]
Moreover $R^K(X,Y,Z)=-R^K(Y,X,Z)$.
\end{prop}

\begin{proof}
Antisymmetry in $(X,Y)$ follows from \eqref{eq:R-nabla} and antisymmetry of the bracket \cite[Proposition~2.31]{Shio}. Additivity follows from \cite[Proposition~3.10(1),(3)]{Shio} and $\RR$-bilinearity of the bracket.

For the first variable, \cite[Propositions~2.31 and 2.32]{Shio} give $[fX,Y]=-(Yf)X+f[X,Y]$. By \cite[Proposition~3.10(2),(4)]{Shio},
\begin{align*}
R^K(fX,Y,Z)
&=f\nabla_X\nabla_YZ-\nabla_Y(f\nabla_XZ)-\nabla_{-(Yf)X+f[X,Y]}Z\\
&=f\nabla_X\nabla_YZ-(Yf)\nabla_XZ-f\nabla_Y\nabla_XZ+(Yf)\nabla_XZ-f\nabla_{[X,Y]}Z
=f\,R^K(X,Y,Z).
\end{align*}
The second variable follows by antisymmetry. For the third,
\[
\nabla_X\nabla_Y(fZ)=(X(Yf))Z+(Yf)\nabla_XZ+(Xf)\nabla_YZ+f\nabla_X\nabla_YZ,
\]
where the two middle terms are symmetric in $(X,Y)$; since $\nabla_{[X,Y]}(fZ)=([X,Y]f)Z+f\nabla_{[X,Y]}Z$ and $[X,Y]f=X(Yf)-Y(Xf)$ \cite[Proposition~2.32]{Shio}, we obtain $R^K(X,Y,fZ)=fR^K(X,Y,Z)$.
\end{proof}

\begin{rem}\label{rem:localize}
Classically, Proposition~\ref{prop:trilinear} implies that $R^K$ is a pointwise object, by localizing with bump functions. This fails here for two independent reasons. A diffeological space may have few smooth functions; for instance every smooth function on the irrational torus $T_\alpha$ is constant, although $T_\alpha$ is elastic and even belongs to $\Elastc$, and carries Riemannian metrics with complete geodesic flow \cite[Corollary~7.1, Example~7.2]{Shio}. And a tangent vector $v\in T_xM$ need not be the value of a global vector field, so an identity in $\XX(M)$ need not say anything about $T_xM$. We therefore formulate the pointwise statement at the level of morphisms, and replace global vector fields by fields along plots; for finitely many tangent vectors at a point this replacement is always available (Lemma~\ref{lem:joint}).
\end{rem}

\subsection{The morphism \texorpdfstring{$\Pi$}{Pi}}

\begin{lem}\label{lem:Pi}
For every $M$ the morphism
\[
\Pi:=\bigl\langle\,T(\pi_M)\circ\pi_{T^2M},\ \ T(\pi_M\circ\pi_{TM}),\ \ \pi_{TM}\circ\pi_{T^2M}\,\bigr\rangle\colon T^3M\longrightarrow T_3M
\]
is well defined, and:
\begin{enumerate}[label=\textup{(\arabic*)}]
\item $\Pi\circ\tau_{TM}=\sigma\circ\Pi$, where $\sigma$ exchanges the first two components of $T_3M$;
\item for a smooth map $f\colon N\to M$, a field $W$ along $f$ and $A,B\in\XX(N)$,
\[
\Pi\circ T^2(W)\circ T(A)\circ B=\bigl\langle T(f)\circ A,\ T(f)\circ B,\ W\bigr\rangle ;
\]
\item in the coordinates of Section~\ref{sec:prelim}, $\Pi\bigl((v_S)_S\bigr)=(u;\,v_{\{1\}},v_{\{2\}},v_{\{0\}})$.
\end{enumerate}
\end{lem}

\begin{proof}
By naturality of $\pi$ the three components lie over $\pi_M\circ\pi_{TM}\circ\pi_{T^2M}$.

(1) Apply \eqref{eq:pi-tau} at $TM$, i.e.\ $\pi_{T^2M}\circ\tau_{TM}=T(\pi_{TM})$ and $T(\pi_{TM})\circ\tau_{TM}=\pi_{T^2M}$. The first component becomes $T(\pi_M)\circ T(\pi_{TM})=T(\pi_M\circ\pi_{TM})$, the second becomes $T(\pi_M)\circ T(\pi_{TM})\circ\tau_{TM}=T(\pi_M)\circ\pi_{T^2M}$, and the third becomes $\pi_{TM}\circ T(\pi_{TM})=\pi_{TM}\circ\pi_{T^2M}$ by naturality of $\pi$.

(2) By naturality of $\pi$, $\pi_{T^2M}\circ T^2(W)=T(W)\circ\pi_{TN}$. Hence the first component is $T(\pi_M\circ W)\circ\pi_{TN}\circ T(A)\circ B=T(f)\circ A\circ\pi_N\circ B=T(f)\circ A$. The second is $T(\pi_M\circ\pi_{TM}\circ T(W))\circ T(A)\circ B=T(f\circ\pi_N\circ A)\circ B=T(f)\circ B$. The third is $\pi_{TM}\circ T(W)\circ\pi_{TN}\circ T(A)\circ B=W\circ\pi_N\circ A\circ\pi_N\circ B=W$.

(3) is immediate.
\end{proof}

\subsection{Joint realization by plots}
Blohmann's axiom (E1) \cite[Definition~2.16]{Shio} states that for every $k\ge1$ the comparison map
\[
\theta_{k,M}\colon\mathbb{L}\hat T_kM=\operatorname{colim}_{P}\hat T_kU_P\longrightarrow T_kM,\qquad
[P,u,a_1,\dots,a_k]\longmapsto\bigl([P,u,a_1],\dots,[P,u,a_k]\bigr),
\]
is a diffeomorphism. Recall that the cocone map $\hat TU_P\to TM$, $(u,a)\mapsto[P,u,a]$, is $T(P)$ \cite[Remark~2.19]{Shio}. We record two consequences.

\begin{lem}[joint realization]\label{lem:joint}
Let $M$ be elastic, $x\in M$, and $v_1,\dots,v_n,z_1,\dots,z_m\in T_xM$ with $n+m\ge1$. Then there are an open neighbourhood $B$ of $0$ in $\RR^n$, a plot $Q\colon B\to M$ with $Q(0)=x$ and $\partial_iQ(0)=v_i$ for $i=1,\dots,n$, and fields $Z_1,\dots,Z_m$ along $Q$ with $Z_k(0)=z_k$ for $k=1,\dots,m$.
\end{lem}

\begin{proof}
By (E1), $\theta_{n+m,M}$ is surjective. Since a colimit in the category of diffeological spaces is a quotient of the coproduct, there are a plot $P\colon U\to M$, a point $u\in U$ and vectors $a_1,\dots,a_n,c_1,\dots,c_m\in\RR^{\dim U}$ with $v_i=[P,u,a_i]$ and $z_k=[P,u,c_k]$; in particular $P(u)=x$. Let $\alpha\colon\RR^n\to\RR^{\dim U}$, $\alpha(t)=u+\sum_it_ia_i$, put $B:=\alpha^{-1}(U)$ and $Q:=P\circ\alpha|_B$. Then $Q(0)=x$, and since $T(\alpha)\circ\partial_i(t)=(\alpha(t),a_i)$,
\[
\partial_iQ(t)=T(P)\bigl(\alpha(t),a_i\bigr)=[P,\alpha(t),a_i],
\]
so $\partial_iQ(0)=v_i$. Put $Z_k(t):=T(P)(\alpha(t),c_k)=[P,\alpha(t),c_k]$. This is smooth, $\pi_M\circ Z_k=P\circ\alpha=Q$ by naturality of $\pi$, and $Z_k(0)=z_k$. (For $n=0$, $B=\RR^0$ and $Q$ is the point $x$.)
\end{proof}

\begin{prop}\label{prop:Pi-subd}
For every elastic space $M$, the morphism $\Pi\colon T^3M\to T_3M$ is a subduction; in particular it is surjective.
\end{prop}

\begin{proof}
Since $\Pi$ is built from components of the natural transformation $\pi$ and from the functor $T$, it is natural in $M$: for every plot $P\colon U\to M$,
\[
\Pi_M\circ T^3(P)=T_3(P)\circ\Pi_U,\qquad T_3(P):=T(P)\times_PT(P)\times_PT(P).
\]
For a domain $U$, Lemma~\ref{lem:Pi}(3) gives $\Pi_U\bigl((v_S)_S\bigr)=(u;v_{\{1\}},v_{\{2\}},v_{\{0\}})$, which has the smooth section $s_U(u;a,b,c):=(v_S)_S$ with $v_\emptyset=u$, $v_{\{0\}}=c$, $v_{\{1\}}=a$, $v_{\{2\}}=b$ and all other components $0$.

Let $\rho\colon W\to T_3M$ be a plot and $w_0\in W$. By (E1), $\theta_{3,M}^{-1}\circ\rho$ is a plot of $\operatorname{colim}_P\hat T_3U_P$. The diffeology of this colimit is the quotient diffeology of the coproduct, and plots of a coproduct are locally plots of one summand. Hence there are an open neighbourhood $W'$ of $w_0$, a plot $P\colon U\to M$ and a smooth map $r\colon W'\to\hat T_3U=T_3U$ with
\[
\rho|_{W'}=\theta_{3,M}\circ[P,r(\cdot)]=T_3(P)\circ r .
\]
Then $\xi:=T^3(P)\circ s_U\circ r\colon W'\to T^3M$ is smooth and
$\Pi_M\circ\xi=T_3(P)\circ\Pi_U\circ s_U\circ r=T_3(P)\circ r=\rho|_{W'}$. Thus every plot of $T_3M$ lifts locally along $\Pi$. Applying this to plots defined on $\RR^0$ shows that $\Pi$ is surjective.
\end{proof}

\begin{dfn}\label{dfn:tensorial}
A vertical connection $K$ is \emph{curvature-tensorial} if there is a morphism $\Rpt=\Rpt^K$ over $M$ such that
\[
\Rpt\colon T_3M\longrightarrow TM,
\qquad C_K=\Rpt\circ\Pi,
\]
and, for every $x\in M$, the restriction $\Rpt_x\colon(T_xM)^3\to T_xM$ is $\RR$-trilinear.
\end{dfn}

\begin{cor}\label{cor:tensorial-unique}
Let $K$ be a vertical connection on an elastic space $M$.
\begin{enumerate}[label=\textup{(\arabic*)}]
\item If $K$ is curvature-tensorial, then $\Rpt^K$ is uniquely determined by $C_K$, and $\Rpt_x$ is antisymmetric in its first two arguments for every $x\in M$.
\item $K$ is curvature-tensorial if and only if $C_K$ is constant on the fibres of $\Pi$ and the induced map $T_3M\to TM$ is fibrewise trilinear; smoothness of the induced map is then automatic.
\end{enumerate}
\end{cor}

\begin{proof}
(1) Uniqueness holds because $\Pi$ is surjective (Proposition~\ref{prop:Pi-subd}). Always $C_K\circ\tau_{TM}=-C_K$, because $\tau_{TM}^2=\id$. By Lemma~\ref{lem:Pi}(1), $\Rpt\circ\sigma\circ\Pi=\Rpt\circ\Pi\circ\tau_{TM}=C_K\circ\tau_{TM}=-\Rpt\circ\Pi$, and surjectivity of $\Pi$ gives $\Rpt\circ\sigma=-\Rpt$.

(2) If $C_K$ is constant on the fibres of $\Pi$, there is a unique map $\Rpt$ with $C_K=\Rpt\circ\Pi$. Since $\Pi$ is a subduction and $C_K$ is smooth, $\Rpt$ is smooth. It lies over $M$ because $C_K$ and $\Pi$ both lie over $\pi_M\circ\pi_{TM}\circ\pi_{T^2M}$.
\end{proof}

\begin{nota}\label{nota:R}
For fields $A,B,C$ along $f\colon N\to M$ we write $\Rpt(A,B,C):=\Rpt\circ\langle A,B,C\rangle$, a field along $f$. In the presence of a metric we put
\[
\Rfour_x(v,w,z,y):=g_x\bigl(\Rpt_x(v,w,z),y\bigr),\qquad \Rfour(A,B,C,D):=g\bigl(\Rpt(A,B,C),D\bigr)\in C^\infty(N,\RR),
\]
and for global vector fields $\Rfour(X,Y,Z,W):=g(R^K(X,Y,Z),W)$, also when $K$ is not curvature-tensorial. Lemma~\ref{lem:compat} shows that the two meanings agree, and that $R^K$, $\Rpt$ and $\Rfour$ have the same order of variables.
\end{nota}

\begin{rem}[notation]\label{rem:R-notation}
The curvature symbols play different roles. The morphism $C_K\colon T^3M\to TM$ is the raw tangent-categorical curvature. When $K$ is curvature-tensorial, its pointwise trilinear factor is
\[
\Rpt^K\colon T_3M\longrightarrow TM.
\]
The expression $R^K(X,Y,Z)$ is the curvature evaluated on global vector fields, while $\Rfour$ denotes the metric contraction. We suppress the superscript $K$ on $\Rpt$ and $\Rfour$ once the connection is fixed.
\end{rem}

\begin{lem}\label{lem:compat}
If $K$ is curvature-tensorial, then $R^K(X,Y,Z)=\Rpt(X,Y,Z)$ for all $X,Y,Z\in\XX(M)$; that is, $R^K(X,Y,Z)(x)=\Rpt_x(X(x),Y(x),Z(x))$.
\end{lem}

\begin{proof}
Lemma~\ref{lem:Pi}(2) with $f=\id_M$, $W=Z$, $A=X$, $B=Y$.
\end{proof}

\subsection{Examples}

\begin{prop}\label{prop:tensorial}
\begin{enumerate}[label=\textup{(\arabic*)}]
\item Every vertical connection on a finite-dimensional manifold is curvature-tensorial, $\Pi$ is a surjective submersion, and $\Rpt_x(v,w,z)$ is the classical curvature $R(v,w)z$ of the covariant derivative $\nabla^K$.
\item Let $\Sigma$ be a diffeological space such that $(\Sigma,T^kN)$ is tangent-commuting for all $k\ge0$ \cite[Convention~7.6]{Shio}, and let $\Xmap=C^\infty(\Sigma,N)$ \textup{(}the source is denoted $M$ in \cite[\S7]{Shio}\textup{)}. If $K_N$ is curvature-tensorial, so is $K_\Xmap=C^\infty(\Sigma,K_N)$, with $\Rpt_\Xmap=C^\infty(\Sigma,\Rpt_N)$; that is, $\Rpt_\Xmap(v,w,z)(p)=\Rpt_N(v(p),w(p),z(p))$ for $p\in\Sigma$.
\end{enumerate}
\end{prop}

\begin{proof}
(1) Both $C_K$ and $\Pi$ are compatible with restriction to open subsets, so we may work in a chart $U\subseteq\RR^n$. We first show that $K=K_\Gamma$ for a smooth family $\Gamma(u)$ of bilinear maps, as in \cite[Example~3.5]{Shio}. Write $K(u,a,b,c)=(u,L_u(a,b,c))$ in the coordinates of that example. By \cite[Remark~3.2]{Shio} and axioms (e), (f), $L_u$ is linear in $(b,c)$ for fixed $a$ and linear in $(a,c)$ for fixed $b$. Moreover, $K\circ0_{TM}=0_M\circ\pi_M$ gives $L_u(a,0,0)=0$, while $K\circ\lambda_M=\id$ gives $L_u(0,0,c)=c$. Therefore
\[
L_u(a,0,c)=L_u(a,0,0)+L_u(0,0,c)=c,
\]
and hence $L_u(a,b,c)=L_u(a,b,0)+L_u(a,0,c)=\Gamma(u)(a,b)+c$, where $\Gamma(u)(a,b):=L_u(a,b,0)$ is bilinear.

Suppressing the base point $u$ and writing $D\Gamma(v)$ for the derivative of $\Gamma$ in the direction $v$, we have
$T(K_\Gamma)(v_S)_S=\bigl(v_{01}+\Gamma(v_0,v_1),\ v_2,\ v_{012}+D\Gamma(v_2)(v_0,v_1)+\Gamma(v_{02},v_1)+\Gamma(v_0,v_{12})\bigr)$, hence
\[
K_\Gamma\circ T(K_\Gamma)(v_S)_S=v_{012}+D\Gamma(v_2)(v_0,v_1)+\Gamma(v_{02},v_1)+\Gamma(v_0,v_{12})+\Gamma\bigl(v_{01}+\Gamma(v_0,v_1),v_2\bigr).
\]
Precomposing with $\tau_{TM}$ exchanges $v_1\leftrightarrow v_2$ and $v_{01}\leftrightarrow v_{02}$. In the difference all terms containing $v_{01},v_{02},v_{12},v_{012}$ cancel, and
\[
C_{K_\Gamma}(v_S)_S=D\Gamma(v_1)(v_0,v_2)-D\Gamma(v_2)(v_0,v_1)+\Gamma\bigl(\Gamma(v_0,v_2),v_1\bigr)-\Gamma\bigl(\Gamma(v_0,v_1),v_2\bigr).
\]
By Lemma~\ref{lem:Pi}(3) this is $\Rpt(v_1,v_2,v_0)$ for a morphism $\Rpt$ that is trilinear and antisymmetric in its first two arguments; comparing with $\nabla^{K_\Gamma}_XY=d\bar Y(\bar X)+\Gamma(\bar Y,\bar X)$ \cite[Example~3.11]{Shio} shows that $\Rpt(v_1,v_2,v_0)$ is the classical $R(v_1,v_2)v_0$. Since $\Pi$ is a surjective submersion, the local morphisms $\Rpt$ are determined by $C_K$ and glue.

(2) By \cite[Proposition~7.9]{Shio}, $C_{K_\Xmap}=C^\infty(\Sigma,C_{K_N})$. Since $\Pi$ is built from structure maps and $C^\infty(\Sigma,-)$ preserves fibre products, $\Pi_\Xmap=C^\infty(\Sigma,\Pi_N)$ under the identifications of \cite[Convention~7.6]{Shio}. Apply $C^\infty(\Sigma,-)$ to $C_{K_N}=\Rpt_N\circ\Pi_N$; trilinearity and antisymmetry hold pointwise in $\Sigma$.
\end{proof}

\section{Covariant derivatives along plots}\label{sec:commute}

The following identity is the analogue of \eqref{eq:R-nabla} for covariant derivatives along a plot. It contains no bracket term, and it evaluates the curvature morphism at tangent vectors that need not extend to global vector fields.

\begin{thm}\label{thm:commute}
Let $K$ be a vertical connection on $M$, let $P\colon U\subseteq\RR^n\to M$ be a plot, let $W$ be a field along $P$ and let $1\le i,j\le n$. Then
\[
\nabla_i\nabla_jW-\nabla_j\nabla_iW=C_K\circ T^2(W)\circ T(\partial_i)\circ\partial_j .
\]
If $K$ is curvature-tensorial, the right-hand side equals $\Rpt(\partial_iP,\partial_jP,W)$.
\end{thm}

\begin{proof}
By functoriality of $T$,
$\nabla_i\nabla_jW=K\circ T\bigl(K\circ T(W)\circ\partial_j\bigr)\circ\partial_i=K\circ T(K)\circ T^2(W)\circ T(\partial_j)\circ\partial_i$.
By \eqref{eq:flip-coord} and naturality of $\tau$ for the morphism $W\colon U\to TM$, i.e.\ $T^2(W)\circ\tau_U=\tau_{TM}\circ T^2(W)$,
\[
\nabla_i\nabla_jW=K\circ T(K)\circ\tau_{TM}\circ T^2(W)\circ T(\partial_i)\circ\partial_j,\qquad
\nabla_j\nabla_iW=K\circ T(K)\circ T^2(W)\circ T(\partial_i)\circ\partial_j .
\]
Subtracting gives the first assertion. The second follows from Lemma~\ref{lem:Pi}(2) with $f=P$, $A=\partial_i$, $B=\partial_j$.
\end{proof}

\begin{rem}
Theorem~\ref{thm:commute} explains the shape of \cite[Definition~5.1]{Shio}: $K\circ T(K)$ is a second covariant derivative, and $\tau_{TM}$ exchanges the two directions of differentiation. Thus $C_K$ is the commutator of covariant derivatives by construction, and \eqref{eq:R-nabla} is the case in which the two directions come from global vector fields, whose bracket need not vanish.
\end{rem}

\begin{cor}\label{cor:plot-sym}
Let $K$ be curvature-tensorial and let $P\colon U\subseteq\RR^n\to M$ be a plot.
\begin{enumerate}[label=\textup{(\arabic*)}]
\item If $K$ is torsion-free, then for all $i,j,k$,
\[
\Rpt(\partial_iP,\partial_jP,\partial_kP)+\Rpt(\partial_jP,\partial_kP,\partial_iP)+\Rpt(\partial_kP,\partial_iP,\partial_jP)=0_M\circ P .
\]
\item If $K$ is a Levi-Civita connection for $g$, then for all $i,j$ and all fields $W,W'$ along $P$,
\[
\Rfour(\partial_iP,\partial_jP,W,W')=-\Rfour(\partial_iP,\partial_jP,W',W).
\]
\end{enumerate}
\end{cor}

\begin{proof}
(1) By Theorem~\ref{thm:commute} and Lemma~\ref{lem:along}(1), the left-hand side equals
\[
\nabla_i(\nabla_j\partial_kP-\nabla_k\partial_jP)+\nabla_j(\nabla_k\partial_iP-\nabla_i\partial_kP)+\nabla_k(\nabla_i\partial_jP-\nabla_j\partial_iP),
\]
which vanishes by Lemma~\ref{lem:symmetry} and Lemma~\ref{lem:along}(1).

(2) Since $\Rpt$ is linear in its third argument and $g$ is bilinear, it suffices by polarization to show $\Rfour(\partial_iP,\partial_jP,W,W)=0$. By \eqref{eq:metric-along} applied twice,
\[
\partial_i\partial_j\,g(W,W)=2g(\nabla_i\nabla_jW,W)+2g(\nabla_jW,\nabla_iW).
\]
Exchanging $i$ and $j$ and subtracting, the left-hand side vanishes because partial derivatives of smooth functions on $U$ commute, and the last term is symmetric. Hence $g(\nabla_i\nabla_jW-\nabla_j\nabla_iW,W)=0$, which is the claim by Theorem~\ref{thm:commute}.
\end{proof}

\section{Symmetries of the curvature}\label{sec:sym}

\subsection{For global vector fields}

\begin{lem}\label{lem:octahedron}
Let $A$ be an abelian group, $B$ a real vector space, and $F\colon A^4\to B$ additive in each variable, satisfying for all $a,b,c,d\in A$
\begin{enumerate}[label=\textup{(\alph*)}]
\item $F(a,b,c,d)=-F(b,a,c,d)$;
\item $F(a,b,c,d)+F(b,c,a,d)+F(c,a,b,d)=0$;
\item $F(a,b,c,d)=-F(a,b,d,c)$.
\end{enumerate}
Then \textup{(d)} $F(a,b,c,d)=F(c,d,a,b)$.
\end{lem}

\begin{proof}
This is the classical argument, which is purely algebraic (cf.\ \cite[\S9]{Mil}). Write (b) for the four quadruples $(a,b,c,d)$, $(b,c,d,a)$, $(c,d,a,b)$, $(d,a,b,c)$ and add the four identities. By (a) and (c), eight of the twelve terms cancel in pairs, and the remaining four give
\[
2F(c,a,b,d)+2F(b,d,a,c)=0 .
\]
By (c) this says $F(c,a,b,d)=F(b,d,c,a)$, and renaming the variables $(c,a,b,d)\mapsto(a,b,c,d)$ gives (d).
\end{proof}

\begin{thm}\label{thm:sym-global}
Let $K$ be a Levi-Civita connection for $g$. For all $X,Y,Z,W\in\XX(M)$:
\begin{enumerate}[label=\textup{(\alph*)}]
\item $\Rfour(X,Y,Z,W)=-\Rfour(Y,X,Z,W)$;
\item $R^K(X,Y,Z)+R^K(Y,Z,X)+R^K(Z,X,Y)=0$;
\item $\Rfour(X,Y,Z,W)=-\Rfour(X,Y,W,Z)$;
\item $\Rfour(X,Y,Z,W)=\Rfour(Z,W,X,Y)$.
\end{enumerate}
Here \textup{(b)} uses only torsion-freeness and \textup{(c)} only metric compatibility.
\end{thm}

\begin{proof}
(a) is Proposition~\ref{prop:trilinear}.

(b) By \eqref{eq:R-nabla} and additivity of $\nabla$ in the second variable, the cyclic sum equals
\[
\nabla_X(\nabla_YZ-\nabla_ZY)+\nabla_Y(\nabla_ZX-\nabla_XZ)+\nabla_Z(\nabla_XY-\nabla_YX)
-\nabla_{[X,Y]}Z-\nabla_{[Y,Z]}X-\nabla_{[Z,X]}Y .
\]
By \cite[Proposition~3.9]{Shio} the first three terms are $\nabla_X[Y,Z]+\nabla_Y[Z,X]+\nabla_Z[X,Y]$, and applying \cite[Proposition~3.9]{Shio} once more the whole expression becomes $[X,[Y,Z]]+[Y,[Z,X]]+[Z,[X,Y]]$, which vanishes by the Jacobi identity \cite[Proposition~2.31]{Shio}. Identity (b) is also a special case of the Bianchi identities established for torsion-free connections in an arbitrary tangent category with negatives \cite[Theorem~3.34]{CC17}; we include the short derivation for completeness.

(c) By polarization it suffices to show $g(R^K(X,Y,Z),Z)=0$. By \cite[Proposition~4.12]{Shio} applied twice,
$X\bigl(Y\,g(Z,Z)\bigr)=2g(\nabla_X\nabla_YZ,Z)+2g(\nabla_YZ,\nabla_XZ)$. Antisymmetrizing in $(X,Y)$ and using \cite[Proposition~2.32]{Shio},
$2g(\nabla_X\nabla_YZ-\nabla_Y\nabla_XZ,Z)=[X,Y]\,g(Z,Z)=2g(\nabla_{[X,Y]}Z,Z)$, and \eqref{eq:R-nabla} gives the claim.

(d) follows from Lemma~\ref{lem:octahedron} with $A=\XX(M)$ and $B=C^\infty(M,\RR)$.
\end{proof}

\begin{rem}
The proof of (c) uses neither definiteness nor non-degeneracy of $g$, and applies verbatim to a weak metric.
\end{rem}

\subsection{Pointwise}

By Remark~\ref{rem:localize}, Theorem~\ref{thm:sym-global} says nothing a priori about $\Rpt_x$. We name the pointwise statement.

\begin{dfn}\label{dfn:S}
A curvature-tensorial Levi-Civita connection $K$ for $g$ has \emph{pointwise curvature symmetries} if, for every $x\in M$, the quadrilinear form $\Rfour_x$ on $T_xM$ satisfies the identities (a)--(d) of Theorem~\ref{thm:sym-global}, with (b) read as $\Rpt_x(v,w,z)+\Rpt_x(w,z,v)+\Rpt_x(z,v,w)=0$.
\end{dfn}

\begin{thm}\label{thm:S}
Every curvature-tensorial Levi-Civita connection $K$ for $g$ on an elastic Riemannian diffeological space $(M,g)$ has pointwise curvature symmetries.
\end{thm}

\begin{proof}
Fix $x\in M$ and $v,w,z,y\in T_xM$. Identity (a) holds by Corollary~\ref{cor:tensorial-unique}(1). For (c), take by Lemma~\ref{lem:joint} a plot $Q$ with $\partial_1Q(0)=v$, $\partial_2Q(0)=w$ and fields $Z,Y$ along $Q$ with $Z(0)=z$, $Y(0)=y$; Corollary~\ref{cor:plot-sym}(2) at $0$ gives $\Rfour_x(v,w,z,y)=-\Rfour_x(v,w,y,z)$. For (b), take by Lemma~\ref{lem:joint} a plot $Q$ with coordinate vectors $v,w,z$ at $0$ and apply Corollary~\ref{cor:plot-sym}(1) at $0$. Identity (d) follows from Lemma~\ref{lem:octahedron} with $A=T_xM$ and $B=\RR$.
\end{proof}

\begin{rem}\label{rem:S-examples}
By Proposition~\ref{prop:tensorial}, Theorem~\ref{thm:S} applies to every Levi-Civita connection on a finite-dimensional manifold, where it recovers the classical symmetries, and to the connection $K_\Xmap$ on a mapping space $\Xmap=C^\infty(\Sigma,N)$ with the metric $g_{\mathit{map}}$ of \cite[Definition~7.11]{Shio}, whenever $\Sigma$ is a closed manifold with a volume form, $(\Sigma,T^kN)$ is tangent-commuting for all $k\ge0$, and $K_N$ is a curvature-tensorial Levi-Civita connection for $g_N$; indeed $K_\Xmap$ is a Levi-Civita connection for $g_{\mathit{map}}$ by \cite[Proposition~7.13]{Shio}. In the latter case one also sees the symmetries directly from
\[
\Rfour_{\Xmap,f}(v,w,z,y)=\int_\Sigma\Rfour_{N,f(p)}\bigl(v(p),w(p),z(p),y(p)\bigr)\,\mathrm{vol}_\Sigma(p),
\]
since (a), (c), (d) hold for the integrand at each $p$, while (b) holds at each $p$ for
\[
\Rpt_\Xmap(v,w,z)(p)=\Rpt_N\bigl(v(p),w(p),z(p)\bigr)
\]
without integration.
\end{rem}

\section{Sectional curvature, the curvature form and Ricci curvature}\label{sec:sec}

\subsection{Sectional curvature and the curvature form}

\begin{lem}\label{lem:prehilbert}
For every $x\in M$, $(T_xM,g_x)$ is a real pre-Hilbert (inner product) space. The formula
\[
\langle v\wedge w,\,z\wedge y\rangle_{\Lambda^2}:=g_x(v,z)\,g_x(w,y)-g_x(v,y)\,g_x(w,z)
\]
extends uniquely to an inner product on the algebraic exterior square $\Lambda^2T_xM$. In particular,
\[
\|v\wedge w\|_{\Lambda^2}^2=\|v\|^2\|w\|^2-g_x(v,w)^2,
\]
which vanishes if and only if $v,w$ are linearly dependent.
\end{lem}

\begin{proof}
The fibre $T_xM$ is a real vector space \cite[Remark~2.18]{Shio}, and $g_x$ is an inner product by \cite[Definition~4.1]{Shio}. The displayed formula is bilinear and alternating in each pair, hence descends to a symmetric bilinear form on $\Lambda^2T_xM$. Every $\alpha\in\Lambda^2T_xM$ belongs to $\Lambda^2F$ for some finite-dimensional subspace $F\subseteq T_xM$. Choosing a $g_x$-orthonormal basis of $F$ identifies the displayed form on $\Lambda^2F$ with the standard Euclidean inner product on the basis $e_i\wedge e_j$ $(i<j)$. Thus it is positive definite on the algebraic exterior square. The final assertion is the equality case of Cauchy--Schwarz.
\end{proof}

\begin{dfn}\label{dfn:curvform}
Let $K$ be a curvature-tensorial Levi-Civita connection for $g$; by Theorem~\ref{thm:S} it has pointwise curvature symmetries. On decomposable elements of the algebraic exterior square set
\[
\Qop_x(v\wedge w,\,z\wedge y):=\Rfour_x(v,w,y,z).
\]
Proposition~\ref{prop:curvform} below shows that this prescription extends to a well-defined symmetric bilinear form on $\Lambda^2T_xM$, which we call the \emph{curvature form}.
For a plane $\sigma\subseteq T_xM$ with basis $(v,w)$, the \emph{sectional curvature} is
\[
\sect_x(\sigma)=\sect_x(v,w):=\frac{\Qop_x(v\wedge w,v\wedge w)}{\|v\wedge w\|_{\Lambda^2}^2}=\frac{\Rfour_x(v,w,w,v)}{\|v\wedge w\|_{\Lambda^2}^2}.
\]
\end{dfn}

\begin{prop}\label{prop:curvform}
$\Qop_x$ is well defined and symmetric, and $\sect_x(\sigma)$ depends only on $\sigma$.
\end{prop}

\begin{proof}
$\Rfour_x$ is quadrilinear, antisymmetric in its first two arguments by (a) and in its last two by (c); hence $\Qop_x$ descends to $\Lambda^2T_xM\times\Lambda^2T_xM$. Symmetry follows from (d) together with (a) and (c). If $(v',w')=(av+cw,\,bv+dw)$ is another basis of $\sigma$, then $v'\wedge w'=(ad-bc)\,v\wedge w$, so numerator and denominator are both multiplied by $(ad-bc)^2$.
\end{proof}

\begin{rem}
Only (a) and (c) enter the definition of $\sect_x$. We use the term \emph{curvature form}, rather than curvature operator, because $\Lambda^2T_xM$ need not be complete. After passing to a Hilbert completion, an operator representing $\Qop_x$ exists only under an additional continuity hypothesis; no such hypothesis is imposed here.
\end{rem}

\subsection{The Ricci curvature}\label{subsec:ricci}
The finite-dimensional formula $\Ric_x(v,w)=\tr\bigl(z\mapsto\Rpt_x(z,v,w)\bigr)$ does not directly make sense on a general pre-Hilbert fibre $T_xM$. A natural Hilbert-space substitute is to pass to the completion $\overline{T_xM}$ and ask whether $z\mapsto\Rpt_x(z,v,w)$ extends to a trace-class operator there. Such an extension need not exist, and even when it is bounded it need not be trace class. For loop groups with Sobolev metrics this phenomenon is well known: the operator $Z\mapsto R(Z,X)Y$ is pseudodifferential of order at most $0$, so the Ricci trace does not exist in the strict sense, and for the $L^2$ metric on $\Omega G$ it diverges \cite[Proposition~1.14, Proposition~2.7]{Fre88}. The following example shows that the same obstruction occurs in the diffeological setting, for a mapping space whose target is not a Lie group.

\begin{ex}[failure of the trace construction]\label{ex:noric}
Let $N=S^2$ with its round metric and Levi-Civita connection, and let $\Xmap=C^\infty(S^1,S^2)$ with the metric $g_{\mathit{map}}$ for the standard volume form on $S^1$ and the connection $K_\Xmap$ of \cite[\S\S7.3--7.4]{Shio}. By Proposition~\ref{prop:tensorial} and \cite[Proposition~7.13]{Shio}, $K_\Xmap$ is a curvature-tensorial Levi-Civita connection for $g_{\mathit{map}}$, hence has pointwise curvature symmetries by Theorem~\ref{thm:S}; moreover, for $f\in\Xmap$ and $z,v,w\in T_f\Xmap=\{v\in C^\infty(S^1,TS^2):\pi\circ v=f\}$,
\[
\Rpt_f(z,v,w)(p)=\langle v(p),w(p)\rangle\,z(p)-\langle z(p),w(p)\rangle\,v(p),\qquad p\in S^1 .
\]
In particular $\Rfour_f(v,w,w,v)=\int_{S^1}\bigl(|v|^2|w|^2-\langle v,w\rangle^2\bigr)d\theta$, so $\sect_f\ge0$ on every plane, with equality for instance when $v$ and $w$ have disjoint supports.

Choose open arcs $J\Subset J'\subseteq S^1$. Since $f^*TS^2$ is trivial over $J'$, choose a smooth unit section $u$ there and extend $u$ by a cut-off to a global smooth field $v\in T_f\Xmap$ with $|v|=1$ on $J$. Put $A:=\Rpt_f(-,v,v)$. The pointwise formula above implies $|A(z)(p)|\le C|z(p)|$ for a constant $C$, so $A$ extends uniquely to a bounded operator on the Hilbert completion $L^2(f^*TS^2)$. For $z$ supported in $J$, $A(z)(p)$ is the orthogonal projection of $z(p)$ onto $v(p)^\perp$.

On $J$ let $e(p):=f(p)\times v(p)$ (cross product in $\RR^3$), a unit section of $v^\perp$. Choose an infinite $L^2(J)$-orthonormal sequence $(\varphi_n)_n\subset C_c^\infty(J)$ and extend each $\varphi_ne$ by zero to $S^1$. These sections lie in $T_f\Xmap$, form an orthonormal sequence in $L^2(f^*TS^2)$, and satisfy $A(\varphi_ne)=\varphi_ne$. Therefore the bounded extension of $A$ is not compact and hence is not trace class. Equivalently, after extending this orthonormal sequence to a Hilbert basis, the diagonal sum for the trace contains infinitely many terms equal to $1$. Thus $\Ric_f(v,v)$ cannot be defined by the Hilbert-space trace.
\end{ex}

\begin{rem}\label{rem:ricci-admissible}
One may call $x$ \emph{Ricci-admissible} if $\Rpt_x(-,v,w)$ extends to a trace-class operator on $\overline{T_xM}$ for all $v,w$, and put $\Ric_x(v,w):=\tr\Rpt_x(-,v,w)$ there. Points with finite-dimensional tangent space are Ricci-admissible, while Example~\ref{ex:noric} shows that no point of $C^\infty(S^1,S^2)$ is. For particular metrics a modified trace may still define a Ricci form; for instance, Freed takes the trace over the Lie algebra first for the $H^{1/2}$ metric on $\Omega G$ \cite[\S\S1--2]{Fre88}.
\end{rem}

\section{Second variation, Jacobi fields and the index form}\label{sec:var}

Throughout this section $(M,g)$ is an elastic Riemannian diffeological space and $K$ is a curvature-tensorial Levi-Civita connection for $g$. By Theorem~\ref{thm:S}, $K$ has pointwise curvature symmetries, so the sectional curvature is defined on every plane. The hypotheses (SH1)--(SH3) of Definition~\ref{dfn:SH} are imposed only where stated: Lemma~\ref{lem:linearize} uses (SH1); Theorem~\ref{thm:jacobi} uses (SH1) and a geodesic flow for existence, and (SH3) for uniqueness; Lemma~\ref{lem:transport} uses a full connection with a geodesic flow; Remark~\ref{rem:variation-pair} and Theorem~\ref{thm:nonmin} use (SH1) and (SH2), but not (SH3). All other results of this section use none of them. For $\gamma\in\Path(M;x,y)$ let
\[
\XX_0(\gamma):=\{W\colon\RR\to TM\ \text{smooth}:\ \pi_M\circ W=\gamma,\ W(0)=0_M(x),\ W(1)=0_M(y)\},
\]
and let $\XX_{00}(\gamma)\subseteq\XX_0(\gamma)$ be the subspace of fields vanishing on $[0,1]$. We write $\nabla_t=\nabla_{\partial}$ for fields along $\gamma$ and $\nabla_t^2=\nabla_t\nabla_t$.

\subsection{The second variation of the energy}

\begin{lem}\label{lem:E-smooth}
$E\colon\Path(M;x,y)\to\RR$ is smooth.
\end{lem}

\begin{proof}
Let $P\colon U\to\Path(M;x,y)$ be a plot and $V_t:=\partial_t\ad(P)$. Then $\ad(P)\colon U\times\RR\to M$ is smooth, hence so is $g(V_t,V_t)=g\circ\langle V_t,V_t\rangle\colon U\times\RR\to\RR$. Since $[0,1]$ is compact, differentiation under the integral sign shows that $u\mapsto E(P(u))=\frac12\int_0^1g(V_t,V_t)(u,t)\,dt$ is smooth.
\end{proof}

\begin{dfn}\label{dfn:index}
Let $\gamma\in\Path(M;x,y)$ satisfy the geodesic equation on $[0,1]$. The \emph{index form} of $\gamma$ is the bilinear form on $\XX_0(\gamma)$ \textup{(}bilinear by Lemma~\ref{lem:along}(1),(3) and trilinearity of $\Rpt$\textup{)}
\[
I(V,W):=\int_0^1\Bigl[g(\nabla_tV,\nabla_tW)-\Rfour(V,\dot\gamma,\dot\gamma,W)\Bigr]dt .
\]
\end{dfn}

By the restriction principle, the integrand on $(0,1)$ depends only on $V|_{(0,1)}$ and $W|_{(0,1)}$; hence $I(V,W)=0$ whenever $V\in\XX_{00}(\gamma)$, and $I$ descends to the quotient $\overline{\XX}_0(\gamma):=\XX_0(\gamma)/\XX_{00}(\gamma)$.

\begin{thm}\label{thm:hessian}
Let $\gamma\in\Path(M;x,y)$ satisfy the geodesic equation on $[0,1]$, and let $P\colon U\to\Path(M;x,y)$ be a $2$-parameter variation plot of $\gamma$ with variation fields $V_1,V_2$ (Definition~\ref{dfn:variation}). Put $V_i^0:=V_i(0,\cdot)$. Then $V_i^0\in\XX_0(\gamma)$ and
\[
\partial_1\partial_2(E\circ P)(0)=I(V_1^0,V_2^0).
\]
\end{thm}

\begin{proof}
Write $F:=\ad(P)\colon U\times\RR\to M$. As in the proof of \cite[Lemma~6.28]{Shio}, for all $s\in U$,
\[
\partial_2(E\circ P)(s)=\int_0^1g(\nabla_2V_t,V_t)(s,t)\,dt .
\]
Differentiating under the integral sign and using \eqref{eq:metric-along} with $A=\partial_1$,
\[
\partial_1\partial_2(E\circ P)=\int_0^1\Bigl[g(\nabla_1\nabla_2V_t,V_t)+g(\nabla_2V_t,\nabla_1V_t)\Bigr]dt .
\]
By Lemma~\ref{lem:symmetry}, $\nabla_2V_t=\nabla_tV_2$ and $\nabla_1V_t=\nabla_tV_1$; by Theorem~\ref{thm:commute} applied to $W=V_2$ in the directions $\partial_1,\partial_t$,
\[
\nabla_1\nabla_2V_t=\nabla_1\nabla_tV_2=\nabla_t\nabla_1V_2+\Rpt(V_1,V_t,V_2).
\]

We evaluate at $s=0$. Let $j_0(t)=(0,t)$. By the restriction principle $V_t\circ j_0=\dot\gamma$, $V_i\circ j_0=V_i^0$, $(\nabla_tV_i)\circ j_0=\nabla_tV_i^0$, and $(\nabla_tV_t)\circ j_0=\nabla_t\dot\gamma$, which vanishes on $[0,1]$. By \eqref{eq:metric-along} with $A=\partial_t$,
\[
\partial_t\,g(\nabla_1V_2,V_t)=g(\nabla_t\nabla_1V_2,V_t)+g(\nabla_1V_2,\nabla_tV_t),
\]
and the last term vanishes along $j_0([0,1])$. Hence
$\int_0^1g(\nabla_t\nabla_1V_2,V_t)(0,t)\,dt=\bigl[g(\nabla_1V_2,V_t)(0,t)\bigr]_{t=0}^{t=1}$. Let $k_0(s)=(s,0)$. Since $F\circ k_0$ is constant at $x$, the restriction principle gives $V_i\circ k_0=0_M(x)$ and $(\nabla_1V_2)\circ k_0=\nabla_1(V_2\circ k_0)=0_M(x)$ by Lemma~\ref{lem:along}(1). The same holds at $t=1$, so the boundary term vanishes; this also shows $V_i^0(0)=0_M(x)$ and $V_i^0(1)=0_M(y)$.

Finally, Corollary~\ref{cor:plot-sym}(2) for the directions $\partial_1,\partial_t$ and the fields $V_2,V_t$ gives $\Rfour(V_1,V_t,V_2,V_t)=-\Rfour(V_1,V_t,V_t,V_2)$. Collecting terms at $s=0$,
\[
\partial_1\partial_2(E\circ P)(0)=\int_0^1\Bigl[g(\nabla_tV_2^0,\nabla_tV_1^0)-\Rfour(V_1^0,\dot\gamma,\dot\gamma,V_2^0)\Bigr]dt=I(V_1^0,V_2^0).\qedhere
\]
\end{proof}

\begin{cor}[second variation]\label{cor:2ndvar}
Let $\gamma\in\Path(M;x,y)$ satisfy the geodesic equation on $[0,1]$, and let $P\colon(-\varepsilon,\varepsilon)\to\Path(M;x,y)$ be a variation plot of $\gamma$ in the sense of \cite[Definition~6.24]{Shio}, with variation field $W:=V_1(0,\cdot)$. Then
\[
\frac{d^2(E\circ P)}{ds^2}(0)=I(W,W)=\int_0^1\Bigl[\|\nabla_tW\|^2-\sect(W,\dot\gamma)\,\|W\wedge\dot\gamma\|_{\Lambda^2}^2\Bigr]dt,
\]
where the second term is read as $0$ at those $t$ at which $W(t),\dot\gamma(t)$ are linearly dependent.
\end{cor}

\begin{proof}
Apply Theorem~\ref{thm:hessian} to the $2$-parameter variation plot $\tilde P(s_1,s_2):=P(s_1+s_2)$, defined on the open neighbourhood $\{(s_1,s_2):s_1+s_2\in(-\varepsilon,\varepsilon)\}$ of $0$. Since $\ad(\tilde P)=\ad(P)\circ(a\times\id)$ with $a(s_1,s_2)=s_1+s_2$ and $T(a)\circ\partial_i=\partial\circ a$, both variation fields of $\tilde P$ at $0$ equal $W$, while $\partial_1\partial_2(E\circ\tilde P)(0)=(E\circ P)''(0)$. The last expression is Definition~\ref{dfn:curvform}; at linearly dependent points $\|W\wedge\dot\gamma\|_{\Lambda^2}^2=0$ and $\Rpt(W,\dot\gamma,\dot\gamma)=0$ by antisymmetry.
\end{proof}

\begin{cor}\label{cor:I-sym}
The index form is symmetric.
\end{cor}

\begin{proof}
By (d), (a) and (c) of Definition~\ref{dfn:S}, $\Rfour(V,\dot\gamma,\dot\gamma,W)=\Rfour(\dot\gamma,W,V,\dot\gamma)=\Rfour(W,\dot\gamma,\dot\gamma,V)$ pointwise.
\end{proof}

The following transport lemma is used to construct variation plots with prescribed variation fields.

\begin{lem}\label{lem:transport}
Let $(K,H)$ be a full connection on $M$ and let $\Phi\colon\RR\times TM\to M$ be a geodesic flow of $(K,H)$. Let $Q\colon U\to M$ be smooth, $U\subseteq\RR^k$ open, and let $Y$ be a field along $Q$. Define
\[
Q^Y\colon U\times\RR\to M,\qquad Q^Y(u,s):=\Phi\bigl(s,Y(u)\bigr),
\]
and let $j\colon U\to U\times\RR$, $j(u)=(u,0)$. Then:
\begin{enumerate}[label=\textup{(\roman*)}]
\item $Q^Y\circ j=Q$ and $\partial_i(Q^Y)\circ j=\partial_iQ$ for $i\le k$;
\item $\partial_s(Q^Y)\circ j=Y$;
\item for every field $Z$ along $Q$, the map
\[
Z^\sharp\colon U\times\RR\to TM,\qquad Z^\sharp(u,s):=T_{(2)}\Phi\bigl(s,\,H(Z(u),Y(u))\bigr)
\]
is a field along $Q^Y$ with $Z^\sharp\circ j=Z$.
\end{enumerate}
\end{lem}

\begin{proof}
(i) $Q^Y(u,0)=\Phi(0,Y(u))=\pi_M(Y(u))=Q(u)$ by \eqref{eq:flow}. Since $T(j)\circ\partial_i=\partial_i\circ j$, we get $T(Q^Y)\circ\partial_i\circ j=T(Q^Y\circ j)\circ\partial_i=T(Q)\circ\partial_i$.

(ii) Since $T$ preserves products and $T(Y)$ maps zero vectors to zero vectors, $T(Q^Y)\circ\partial_s(u,s)=T(\Phi)\bigl(\partial(s),0_{TM}(Y(u))\bigr)=\Phi^{(1)}(s,Y(u))$, which equals $Y(u)$ at $s=0$ by \eqref{eq:flow}.

(iii) $H(Z(u),Y(u))$ is defined because $\pi_M\circ Z=Q=\pi_M\circ Y$, and $Z^\sharp$ is smooth. Write $\Phi_s:=\Phi(s,-)$; by \eqref{eq:partial-tangent}, $Z^\sharp(u,s)=T(\Phi_s)\bigl(H(Z(u),Y(u))\bigr)$. By naturality of $\pi$ and \eqref{eq:H},
\[
\pi_M\bigl(Z^\sharp(u,s)\bigr)=\Phi_s\bigl(\pi_{TM}H(Z(u),Y(u))\bigr)=\Phi(s,Y(u)).
\] Finally $\Phi_0=\pi_M$ by \eqref{eq:flow}, so $Z^\sharp(u,0)=T(\pi_M)\bigl(H(Z(u),Y(u))\bigr)=Z(u)$ by \eqref{eq:H}.
\end{proof}

\begin{rem}\label{rem:variation-pair}
Corollary~\ref{cor:I-sym} can also be read off from Theorem~\ref{thm:hessian} and the symmetry of mixed partial derivatives of $E\circ P$, provided every pair $V,W\in\XX_0(\gamma)$ occurs as the pair of variation fields of a $2$-parameter variation plot. Under (SH1) and (SH2), such a variation plot is given by
\[
\ad(P)(s_1,s_2,t):=\Phi\Bigl(s_1,\ T_{(2)}\Phi\bigl(s_2,H(V(t),W(t))\bigr)\Bigr),
\]
as one checks by applying Lemma~\ref{lem:transport} first to $Q=\gamma$ and $Y=W$ and then to $Q^W$ and the field $V^\sharp$; the endpoints are fixed because $H(0_M(x),0_M(x))=0_{TM}(0_M(x))$ by the scalar axiom of $H$ \cite[Definition~3.12(e)]{Shio}, $T(\Phi_s)\circ0_{TM}=0_M\circ\Phi_s$ by naturality of $0$, and $\Phi(s,0_M(x))=x$ by (SH2).
\end{rem}

\subsection{Jacobi fields}

\begin{dfn}\label{dfn:jacobi}
Let $\gamma\colon\RR\to M$ be a path and $[a,b]\subseteq\RR$ a closed interval. A field $J$ along $\gamma$ \emph{satisfies the Jacobi equation on $[a,b]$} if
\[
\bigl(\nabla_t^2J+\Rpt(J,\dot\gamma,\dot\gamma)\bigr)(t)=0_M(\gamma(t))\qquad\text{for all }t\in[a,b].
\]
If $\gamma$ is a geodesic, a \emph{Jacobi field} along $\gamma$ is a field $J$ along $\gamma$ for which this holds for all $t\in\RR$.
\end{dfn}

\begin{prop}\label{prop:jacobi-var}
Let $F\colon(-\varepsilon,\varepsilon)\times\RR\to M$ be smooth with $F(s,\cdot)$ a geodesic for every $s$. Then $J:=\partial_sF(0,\cdot)$ is a Jacobi field along $\gamma:=F(0,\cdot)$.
\end{prop}

\begin{proof}
Write $V_s=\partial_sF$ and $V_t=\partial_tF$. By the restriction principle applied to the slices $t\mapsto(s,t)$, $\nabla_tV_t=0_M\circ F$. Theorem~\ref{thm:commute} with $W=V_t$ and Lemma~\ref{lem:along}(1) give $-\nabla_t\nabla_sV_t=\Rpt(V_s,V_t,V_t)$, and $\nabla_sV_t=\nabla_tV_s$ by Lemma~\ref{lem:symmetry}. Hence $\nabla_t^2V_s+\Rpt(V_s,V_t,V_t)=0$; restrict to $s=0$.
\end{proof}

The Jacobi equation is the linearization of the geodesic equation. To make this precise, let
\[
S^\uparrow:=\tau_{TM}\circ T(S)\colon T^2M\to T^3M .
\]
By \eqref{eq:pi-tau} at $TM$, $\pi_{T^2M}\circ S^\uparrow=T(\pi_{TM}\circ S)=\id_{T^2M}$, so $S^\uparrow$ is a vector field on $T^2M$, the tangent lift of $S$.

\begin{lem}\label{lem:linearize}
Assume \textup{(SH1)}. Let $\gamma$ be a geodesic, $J$ a field along $\gamma$, and
\[
\tilde J:=\Theta^{-1}\circ\langle\dot\gamma,\ J,\ \nabla_tJ\rangle\colon\RR\to T^2M .
\]
Then for every $t\in\RR$, $T(\tilde J)(\partial(t))=S^\uparrow(\tilde J(t))$ if and only if $\bigl(\nabla_t^2J+\Rpt(J,\dot\gamma,\dot\gamma)\bigr)(t)=0$.
\end{lem}

\begin{proof}
By construction $\pi_{TM}\circ\tilde J=\dot\gamma$, $T(\pi_M)\circ\tilde J=J$ and $K\circ\tilde J=\nabla_tJ$. Put $\mathsf A:=T(\tilde J)\circ\partial$ and $\mathsf B:=S^\uparrow\circ\tilde J$. Since $\Theta$ is a diffeomorphism and $T$ preserves the fibre product $T_3M$, the morphism $T(\Theta)=\langle T(\pi_{TM}),T^2(\pi_M),T(K)\rangle$ is a diffeomorphism; so $\mathsf A(t)=\mathsf B(t)$ if and only if the images under $T(\pi_{TM})$, $T^2(\pi_M)$ and $T(K)$ agree, and elements of $T^2M$ are compared through $\Theta$. We use \eqref{eq:pi-tau} at $TM$ throughout.

\emph{First component.} $T(\pi_{TM})\circ\mathsf A=T(\dot\gamma)\circ\partial$ and $T(\pi_{TM})\circ\mathsf B=\pi_{T^2M}\circ T(S)\circ\tilde J=S\circ\dot\gamma$. Both have $\Theta$-components $(\dot\gamma,\dot\gamma,0)$, since $\gamma$ is a geodesic and by \eqref{eq:H}.

\emph{Second component.} $T^2(\pi_M)\circ\mathsf A=T(J)\circ\partial$, with $\Theta$-components $(J,\dot\gamma,\nabla_tJ)$. By naturality of $\tau$, $T^2(\pi_M)\circ\mathsf B=\tau_M\circ T(T(\pi_M)\circ S)\circ\tilde J=\tau_M\circ\tilde J$, whose $\Theta$-components are $(T(\pi_M)\tilde J,\pi_{TM}\tilde J,K\tau_M\tilde J)=(J,\dot\gamma,\nabla_tJ)$ by \eqref{eq:pi-tau} and torsion-freeness.

\emph{Third component.} $T(K)\circ\mathsf A=T(\nabla_tJ)\circ\partial$, with $\Theta$-components $(\nabla_tJ,\dot\gamma,\nabla_t^2J)$. For $T(K)\circ\mathsf B$: first, $\pi_{TM}\circ T(K)\circ\mathsf B=K\circ\pi_{T^2M}\circ\tau_{TM}\circ T(S)\circ\tilde J=K\circ T(\pi_{TM}\circ S)\circ\tilde J=\nabla_tJ$. Second, by axiom (b) of \cite[Definition~3.1]{Shio}, $T(\pi_M)\circ T(K)\circ\mathsf B=T(\pi_M)\circ T(\pi_{TM})\circ\tau_{TM}\circ T(S)\circ\tilde J=T(\pi_M)\circ S\circ\dot\gamma=\dot\gamma$. Third, by the definition of $C_K$,
\[
K\circ T(K)\circ\tau_{TM}\circ T(S)\circ\tilde J=K\circ T(K\circ S)\circ\tilde J+C_K\circ T(S)\circ\tilde J .
\]
Since $K\circ S=0_M\circ\pi_M$, the first term is $K\circ T(0_M)\circ J$, which vanishes by \eqref{eq:K-zero}. For the second, naturality of $\pi$ gives $\Pi\circ T(S)=\langle\pi_{TM},T(\pi_M),\pi_{TM}\rangle$, so $C_K\circ T(S)\circ\tilde J=\Rpt(\dot\gamma,J,\dot\gamma)=-\Rpt(J,\dot\gamma,\dot\gamma)$. Hence $T(K)\circ\mathsf B$ has $\Theta$-components $(\nabla_tJ,\dot\gamma,-\Rpt(J,\dot\gamma,\dot\gamma))$.

Comparing the three components proves the lemma.
\end{proof}

\begin{thm}\label{thm:jacobi}
Assume \textup{(SH1)}, and let $\Phi$ be a geodesic flow of $(K,H)$. Let $v_0\in T_xM$ and $\gamma:=\Phi(\cdot,v_0)$. For all $a,b\in T_xM$ there is a Jacobi field $J$ along $\gamma$ with $J(0)=a$ and $\nabla_tJ(0)=b$. If moreover \textup{(SH3)} holds, then $J$ is unique, and every geodesic $\gamma$ of $K$ with $\gamma(0)=x$ is of the form $\Phi(\cdot,v_0)$ with $v_0=\dot\gamma(0)\in T_xM$; in particular \textup{(SH2)} holds.
\end{thm}

\begin{proof}
\emph{Existence.} Let $\xi:=\Theta^{-1}(v_0,a,b)\in T^2M$. Since $\delta_{TM}$ is a subduction \cite[Proposition~2.14]{Shio}, there is a smooth curve $c\colon\RR\to TM$ with $\delta(c)=T(c)(\partial(0))=\xi$; then $c(0)=\pi_{TM}(\xi)=v_0$. Put $F(s,t):=\Phi(t,c(s))$. Each $F(s,\cdot)$ is a geodesic by \eqref{eq:flow}, and $F(0,\cdot)=\gamma$, so $J:=\partial_sF(0,\cdot)$ is a Jacobi field by Proposition~\ref{prop:jacobi-var}. Let $k(s)=(s,0)$. Then $F\circ k=\pi_M\circ c$ by \eqref{eq:flow}, so $J(0)=T(\pi_M)(\delta(c))=T(\pi_M)(\xi)=a$. Moreover $\partial_tF\circ k=\Phi^{(1)}(0,c(\cdot))=c$ by \eqref{eq:flow}, so Lemma~\ref{lem:symmetry} and the restriction principle give
\[
\nabla_tJ(0)=(\nabla_s\partial_tF)(0,0)=\nabla_\partial c\,(0)=K(\delta(c))=K(\xi)=b .
\]

\emph{Uniqueness.} Let $J,J'$ be Jacobi fields with the same initial data. By Lemma~\ref{lem:linearize}, $\tilde J$ and $\tilde J'$ are solutions of the dynamical system $(T^2M,S^\uparrow,\Theta^{-1}(v_0,a,b))$ with time domain $(\RR,\partial,0)$ in the sense of \cite{CCL} and \cite[\S6.1]{Shio}. By (SH3), $(\RR,\partial,0)$ is preinitial for $T^2M$ \cite[Definition~6.20]{Shio}, so $\tilde J=\tilde J'$, and $J=T(\pi_M)\circ\tilde J=J'$.

\emph{Every geodesic is a flow line.} Since $(K,H)$ has a geodesic flow, $S$ is complete (Section~\ref{sec:prelim}). Under (SH3), \cite[Proposition~6.23]{Shio} therefore shows that $\Phi$ fixes the zero section and that every geodesic $\gamma$ of $K$ satisfies $\gamma=\Phi(\cdot,\dot\gamma(0))$.
\end{proof}

\begin{rem}\label{rem:why-all-n}
Theorem~\ref{thm:jacobi} makes the role of two structural hypotheses of \cite{Shio} explicit. Effectiveness of $K$ is what allows the initial data $(J(0),\nabla_tJ(0))$ to be prescribed arbitrarily. The requirement in \cite[Definition~6.20]{Shio} that $(\RR,\partial,0)$ be preinitial for \emph{every} iterated tangent bundle, motivated in \cite[Remark~6.21]{Shio} by stability under $T$, is what gives uniqueness, because the linearized geodesic equation lives on $T^2M$.
\end{rem}

\subsection{The null space of the index form and conjugate points}

\begin{prop}\label{prop:nullspace}
Let $\gamma\in\Path(M;x,y)$ satisfy the geodesic equation on $[0,1]$. For every smooth field $J$ along $\gamma$ and every $W\in\XX_0(\gamma)$,
\[
\int_0^1\Bigl[g(\nabla_tJ,\nabla_tW)-\Rfour(J,\dot\gamma,\dot\gamma,W)\Bigr]dt=-\int_0^1g\bigl(\nabla_t^2J+\Rpt(J,\dot\gamma,\dot\gamma),\,W\bigr)dt .
\]
Consequently, $J\in\XX_0(\gamma)$ satisfies $I(J,W)=0$ for all $W\in\XX_0(\gamma)$ if and only if $J$ satisfies the Jacobi equation on $[0,1]$.
\end{prop}

\begin{proof}
By \eqref{eq:metric-along}, $\frac{d}{dt}g(\nabla_tJ,W)=g(\nabla_t^2J,W)+g(\nabla_tJ,\nabla_tW)$, and the boundary term vanishes since $W(0)$ and $W(1)$ are zero vectors. This proves the identity and the ``if'' part. Conversely, let $Z:=\nabla_t^2J+\Rpt(J,\dot\gamma,\dot\gamma)$ and suppose $Z(t_0)\neq0$ for some $t_0\in(0,1)$. Then $h:=g(Z,Z)>0$ on an interval $(a,b)\ni t_0$ contained in $(0,1)$. For $\varphi\in C^\infty(\RR,\RR)$ with $\varphi\ge0$, $\supp\varphi\subset(a,b)$ and $\varphi(t_0)>0$, the field $W:=\varphi Z$ lies in $\XX_0(\gamma)$ and $0=I(J,W)=-\int_0^1\varphi h\,dt<0$, a contradiction. Hence $Z=0$ on $(0,1)$, and on $[0,1]$ by continuity of $h$ and definiteness of $g$.
\end{proof}

\begin{dfn}\label{dfn:conjugate}
Let $\gamma\in\Path(M;x,y)$ satisfy the geodesic equation on $[0,1]$, and let $0<t_0\le1$. The point $\gamma(t_0)$ is \emph{conjugate} to $\gamma(0)$ along $\gamma$ if there is a field $J$ along $\gamma$ satisfying the Jacobi equation on $[0,t_0]$ with $J(0)=0$, $J(t_0)=0$ and $J|_{[0,t_0]}\not\equiv0$.
\end{dfn}

\begin{cor}\label{cor:conj}
Let $\gamma\in\Path(M;x,y)$ satisfy the geodesic equation on $[0,1]$. The form induced by $I$ on $\overline{\XX}_0(\gamma)$ is weakly non-degenerate if and only if $\gamma(1)$ is not conjugate to $\gamma(0)$ along $\gamma$.
\end{cor}

\begin{proof}
By Proposition~\ref{prop:nullspace}, the null space of $I$ consists of the fields in $\XX_0(\gamma)$ that satisfy the Jacobi equation on $[0,1]$; such a field is non-zero in $\overline{\XX}_0(\gamma)$ exactly when it does not vanish identically on $[0,1]$.
\end{proof}

\begin{rem}
Every Jacobi field in the sense of Theorem~\ref{thm:jacobi} satisfies the Jacobi equation on $[0,1]$. The converse, extending a field satisfying the Jacobi equation on $[0,1]$ to a Jacobi field along $\gamma$, would require a local uniqueness statement that the curve-object property of $\Elastc$ does not provide.
\end{rem}

\subsection{Non-positive sectional curvature}

\begin{thm}\label{thm:nonpos}
Suppose $\sect_x(\sigma)\le0$ for every $x\in M$ and every plane $\sigma\subseteq T_xM$, and let $\gamma\in\Path(M;x,y)$ satisfy the geodesic equation on $[0,1]$.
\begin{enumerate}[label=\textup{(\arabic*)}]
\item For every $W\in\XX_0(\gamma)$, $I(W,W)\ge\int_0^1\|\nabla_tW\|^2dt\ge0$, with $I(W,W)=0$ if and only if $W\in\XX_{00}(\gamma)$. In particular the form induced by $I$ on $\overline{\XX}_0(\gamma)$ is positive definite.
\item No point $\gamma(t_0)$ with $0<t_0\le1$ is conjugate to $\gamma(0)$ along $\gamma$.
\end{enumerate}
\end{thm}

\begin{proof}
(1) The inequality follows from Corollary~\ref{cor:2ndvar}. If $I(W,W)=0$, then $\nabla_tW=0$ on $[0,1]$, so $\frac{d}{dt}g(W,W)=2g(\nabla_tW,W)=0$ there; since $W(0)=0$, definiteness gives $W=0$ on $[0,1]$.

(2) Let $J$ be a field along $\gamma$ satisfying the Jacobi equation on $[0,t_0]$ with $J(0)=J(t_0)=0$, and $f:=\frac12g(J,J)$. By \eqref{eq:metric-along} and the Jacobi equation, on $[0,t_0]$,
\[
f''=\|\nabla_tJ\|^2-\Rfour(J,\dot\gamma,\dot\gamma,J)\ge0 .
\]
Indeed, when $J(t)$ and $\dot\gamma(t)$ are linearly independent, the curvature term is $\sect(J,\dot\gamma)\|J\wedge\dot\gamma\|_{\Lambda^2}^2\le0$, while it is $0$ when they are dependent by antisymmetry of $\Rpt$. Thus $f\ge0$ is convex on $[0,t_0]$ and vanishes at both ends, so $f=0$ and $J=0$ on $[0,t_0]$.
\end{proof}

\subsection{Paths that do not minimize length}

\begin{thm}\label{thm:nonmin}
Assume \textup{(SH1)} and \textup{(SH2)}. Let $\gamma\in\Path(M;x,y)$ satisfy the geodesic equation on $[0,1]$. If $I(W,W)<0$ for some $W\in\XX_0(\gamma)$, then there is $\sigma\in\Path(M;x,y)$ with $L(\sigma)<L(\gamma)$.
\end{thm}

\begin{proof}
The hypotheses are the standing assumptions of \cite[\S6.6]{Shio}. Let $P_W$ be the variation plot of \cite[Lemma~6.32]{Shio}, $\ad(P_W)(s,t)=\Phi(s,W(t))$, whose variation field at $s=0$ is $W$. By \cite[Proposition~6.30]{Shio}, $(E\circ P_W)'(0)=0$, and by Corollary~\ref{cor:2ndvar}, $(E\circ P_W)''(0)=I(W,W)<0$. Hence $E(P_W(s_0))<E(\gamma)$ for some $s_0$. The speed of $\gamma$ is constant on $[0,1]$, since $\frac{d}{dt}g(\dot\gamma,\dot\gamma)=2g(\nabla_t\dot\gamma,\dot\gamma)=0$ there, so $E(\gamma)=\frac12L(\gamma)^2$ by \eqref{eq:LE}. For $\sigma:=P_W(s_0)$, again by \eqref{eq:LE},
$L(\sigma)^2\le2E(\sigma)<2E(\gamma)=L(\gamma)^2$.
\end{proof}

\begin{rem}
Theorem~\ref{thm:nonmin} complements \cite[Proposition~6.37]{Shio}: under the same hypotheses, a regular length-minimizing path has a constant-speed reparametrization that satisfies the geodesic equation on $[0,1]$. Together with Theorem~\ref{thm:nonpos}, this gives the expected variational implication that a negative index direction obstructs minimality, whereas non-positive sectional curvature rules out such directions.
\end{rem}

\begin{rem}[mapping spaces]\label{rem:mapping-var}
Let $\Sigma$ be a closed manifold with a volume form, let $(N,g_N)$ be a geodesically complete finite-dimensional Riemannian manifold with the full connection $(K_N,H_N)$ given by its Levi-Civita connection, and let $\Xmap=C^\infty(\Sigma,N)$ carry the metric $g_{\mathit{map}}$ and the full connection $(K_\Xmap,H_\Xmap)$ of \cite[\S\S7.3--7.4]{Shio}; the tangent-commuting hypothesis holds because $N$ is a manifold \cite[Convention~7.6]{Shio}. Then $\Xmap$ satisfies (SH1) and (SH2): $K_\Xmap$ is a Levi-Civita connection for $g_{\mathit{map}}$ \cite[Proposition~7.13]{Shio} and is curvature-tensorial by Proposition~\ref{prop:tensorial}, and the geodesic flow $\Phi_\Xmap$ of \cite[Proposition~7.15]{Shio} fixes the zero section by \cite[Theorem~7.17]{Shio}, since $N\in\Elastc$ \cite[Remark~6.21]{Shio}. Consequently Theorem~\ref{thm:nonmin}, Remark~\ref{rem:variation-pair} and the existence part of Theorem~\ref{thm:jacobi} apply to $\Xmap$, for instance to the free loop space $C^\infty(S^1,N)$. Whether $\Xmap$ satisfies (SH3), that is, whether $\Xmap\in\Elastc$, is open \cite[Remark~7.16]{Shio}, so the uniqueness part of Theorem~\ref{thm:jacobi} is not obtained in this way.
\end{rem}


\end{document}